\documentclass[a4paper,12 pt]{article}
\usepackage[utf8]{inputenc}
\usepackage{amsmath}
\usepackage{amssymb}
\usepackage{latexsym}
\usepackage{amsthm}
\usepackage{amsfonts}
\usepackage{graphicx}
\usepackage{color}
\usepackage{shuffle}

\newtheorem{theorem}{Theorem}[section]
\newtheorem{lemma}{Lemma}[section]
\newtheorem{proposition}{Proposition}[section]

\def \a{\alpha}
\def \z{\zeta}
\def\k{\textbf{k}}

\def\hh{\mathfrak{h}}
\def\SH{\shuffle}

\def\SH{\shuffle}

\usepackage[a4paper,top=3cm,bottom=2cm,left=1.5cm,right=1.5cm,marginparwidth=1cm]{geometry}
\title{An algebraic proof of the duality of multiple zeta-star values of height one}
\author{Nita Tamang$^\dag$, Pitu Sarkar$^\ddag$ \\ Department of Mathematics, University of North Bengal$^{\dag,\ddag}$ \\ West Bengal, India 734013 \\ nita\_math@nbu.ac.in$^\dag$, pitucob2016@gmail.com$^\ddag$}
\date{}

\begin{document}
	\maketitle

\begin{center}
	\textbf{Abstract}
\end{center}
Shuffle algebra has been employed to give a proof of the duality theorem for multiple zeta-star values of height one.

\vspace{.8cm}
\hspace*{.66cm}\textbf{Mathematics Subject Classification: 11M32.} 	
	
	\section{Introduction}
	
	After over two centuries of oblivion,  M. E. Hoffman \cite{10} and D. Zagier \cite{38} independently rediscovered multiple zeta values in the 1990s. It was first studied in 1775 by L. Euler in a special case \cite{23}. These values appear in various contexts in mathematics and physics, such as Combinatorics, Knot theory, Mirror symmetry, Motivic theory, Feynman integrals, Witten’s zeta functions, Kontsevich’s deformation quantization, etc. The last two decades saw extensive work by various mathematicians and physicists, including Brown, Broadhurst, Kreimer Deligne, Ecalle, Cartier, Drinfeld, Hoffman, Goncharov, Hain, Kontsevich, Terasoma, Zagier and many others. Study of Multiple zeta values are nowadays an active and fast-moving field of research.
	
	A finite sequence of positive integers $\k=\left(k_1,\dots,k_n\right)$ is called an index. We call an index  $\k$  admissible when $k_1 \geq 2$.
	For an admissible index $\k=\left(k_1,\dots,k_n\right)$,
	the multiple zeta value (MZV) $\zeta(\k)$ and the multiple zeta-star value(MZSV) $\zeta^*(\k)$  are   defined respectively by \\
	$$\zeta(\k)= \zeta \left( k_1,\dots,k_n\right)=\sum_{m_1>\dots >m_n\geq 1}\dfrac{1}{m_1^{k_1}\dots m_n^{k_n}},$$
	
	and $$\zeta^{*}(\k)= \zeta^{*}\left( k_1,\dots,k_n\right)=\sum_{m_1\geq\dots\geq m_n\geq 1}\dfrac{1}{m_1^{k_1}\dots m_n^{k_n}}.$$

	
	For an index $\k=\left(k_1,\dots,k_n\right)$, weight, depth and height of $\z(\k)$ or $\z^*(\k) $  are respectively given by
	$$\text{wt}(\k)=k_1+\dots+k_n,~\text{dep}(\k)=n,~\text{and} ~\text{ht}(\k)=\#\{l:1 \leq l \leq n, k_l \geq 2\}.$$

	Among numerous results, an important and an interesting one in the study of MZV's is the duality formula. Any multi-index $\k$ is written in the form $$\k=(a_1+1,\underbrace{1, \dots , 1}_{b_1-1},a_2+1,\underbrace{1, \dots , 1}_{b_2-1}, \dots ,a_u+1,\underbrace{1, \dots , 1}_{b_u-1}), $$
	where $a_1,b_1, \dots , a_u,b_u$ are all positive integers. The dual index $\k'$ of $\k$ is defined by
	$$\k'=(b_u+1,\underbrace{1, \dots , 1}_{a_u-1},b_{u-1}+1,\underbrace{1, \dots , 1}_{a_{u-1}-1}, \dots ,b_1+1,\underbrace{1, \dots , 1}_{a_1-1}). $$
	
	Then the duality formula for multiple zeta values states as:
	
	\begin{theorem}
		\label{thm1}
		(Duality formula) For an index $\k$ and its dual $\k'$, the following equality holds:
		$$\z(\k)=\z(\k').$$
	\end{theorem}

	In case of height one, the formula reduces to
	$$\z(k+1,\underbrace{1, \dots ,1}_{n-1})=\z(n+1,\underbrace{1, \dots ,1}_{k-1})$$ for $k,n \geq 1$.
	Duality formula was proved in \cite{38} by  D. Zagier  using the iterated integral representation of multiple zeta values and was also proved as a consequence of regularized double shuffle relation by Jun Kajikawa in \cite{39}.
	Many researchers tried to find a similar result for multiple zeta-star values but nothing of that sort could be proved until 2007, when
	M. Kaneko conjectured, and Y. Ohno proved the following theorem, which is a kind of duality for multiple zeta-star values of height one.
	\begin{theorem}
		\label{T2}	For any integers $k,n \geq 1$, we have
		$${(-1)^k}\z^*(k+1,\underbrace{1, \dots ,1}_{n})-{(-1)^n}\z^*(n+1,\underbrace{1, \dots ,1}_{k}) \in \mathbb{Q}[\z(2),\z(3),\z(5), \dots],$$ the right-hand side being the algebra over $\mathbb{Q}$ generated by the values of the Riemann zeta function at positive integer arguments($>1$).
	\end{theorem}
	
Due to the conjecture that all relations among multiple zeta values will follow 
from regularized double shuffle relation \cite{30}, a good number of results in the study of multiple zeta values and multiple zeta star values, have been proved using algebraic approach, namely, sum formula, cyclic sum formula, duality of multiple zeta values etc. Also, the literature lacks the algebraic proof of duality of multiple zeta star values (Theorem \ref{T2}), though  the result has been proved already in \cite{34}, by M. Kaneko and Y. Ohno, using an integral function defined by T. Arakawa and M. Kaneko \cite{101} and proved in \cite{35},  using a generating function for sums of multiple zeta values studied by Aoki, Kombu and Ohno \cite{100}. The purpose of this article is to provide an alternative proof of Theorem \ref{T2} using an algebraic approach. In particular, Theorem 1.2 has been proved with the help of shuffle algebra.	
	
	Now we state an ingredient of our proof of Theorem \ref{T2}, namely, Ohno's relation  which generalizes simultaneously the sum formula \cite{10} and the duality theorem:
	
	\begin{theorem} \cite{36}
		\label{ohn}
		Let $(k_1, \dots k_n)$ be an admissible index and $(k'_1, \dots k'_{n'})$ the dual index of $(k_1, \dots k_n)$. Then
		\begin{align*}\
			\sum_{\substack{l_1+ \dots +l_n=l \\l_i \geq 0}} \zeta(k_1+l_1,\dots ,k_n+l_n)&=\sum_{ \substack{l_1+ \dots +l_{n'}=l \\l_i \geq 0}} \zeta(k'_1+l_1,\dots ,k'_{n'}+l_{n'})
		\end{align*}
		holds for any integer $l \geq 0$.
	\end{theorem}


	\textbf{Shuffle Algebra:}	
	\cite{31,32} Let $A=\{x,y\}$ be an alphabet which contains two noncommutative letters. Denote by $A^*$  the set of all words generated by the letters in A and the empty word 1. The subset of the nonempty words of $A^*$ is denoted by $A^+$.
	Let $\hh=\mathbb{Q}<A>$ be the $\mathbb{Q}$-algebra of noncommutative polynomials on $x$ and $y$ with coefficients in $\mathbb{Q}$. Then $A^*$ is a basis of the $\mathbb{Q}$-vector space $\hh$. Define two subalgebras of $\hh$,~$$\hh^1=\mathbb{Q}+\hh y,~ \hh^0= \mathbb{Q}+x\hh y$$ which have the  $\mathbb{Q}$-basis
	$$z_{k_1} \dots z_{k_n}, n \geq 0, k_1, \dots , k_n \geq 1$$
	and $$z_{k_1} \dots z_{k_n}, n \geq 0, k_1 \geq 2, k_2, \dots , k_n \geq 1,$$
	respectively, where $z_k=x^{k-1}y$ for any positive integer $k$.

		Define the shuffle product $\SH$ on $\hh$ by $ \mathbb{Q}$-bilinearities and the axioms: \\
		(S1) $1 \SH w =w \SH 1 =w,$  \hspace{3em} ($\forall w \in A^*$);\\
		(S2) $aw_1 \SH bw_2 = a(w_1 \SH bw_2) + b(aw_1 \SH  w_2) $, ~ ( $\forall  a,b \in{A},\forall w_1,w_2 \in A^*$).
		

  The product $\SH$ is commutative and associative. Under this product $\SH$, $\hh$ becomes a commutative algebra, and  $\hh^1$ and
	$\hh^0$ are its subalgebras. These commutative algebras are called shuffle  algebras, and denoted by   $\hh_{\SH}$,  $\hh^1_{\SH}$ and  $\hh^0_{\SH}$, respectively.
	Let $\sigma$ be an automorphism of the noncommutative algebra $\hh$, determined by  $$\sigma(x) = x,~\sigma(y)=x+y.$$
	The inverse map $\sigma^{-1}$ is determined by
	$$\sigma^{-1}(x) = x,~\sigma^{-1}(y)=-x+y.$$\\
	Note that  $\sigma$  and  $\sigma^{-1}$ are also automorphisms of the shuffle algebra $\hh_{\SH}$.
	Now, define the $\mathbb{Q}$-linear map $S:\hh \rightarrow \hh$ by $S(1)=1$ and
	$$S(wa)=\sigma(w)a,~\text{(for all } w\in{A^*,\text{for all }a \in{A})}.$$
	Note that $S$ is invertible, and $S^{-1}(1)=1$ and
	$$S^{-1}(wa)=\sigma^{-1}(w)a,~\text{(for all } w\in{A^*,\text{for all }a \in{A})}.$$
	
	Now for any $\mathbb{Q}$-linear map $Z:\hh^0 \rightarrow \mathbb{R}$, define
	$$Z^*=Z \circ S : \hh^0 \rightarrow \mathbb{R},$$
	which is also a $\mathbb{Q}$-linear map.

	Then for any admissible index $\k=(k_1,\dots,k_n)$, the multiple zeta value $\z(\k)$ and the multiple zeta-star value $\z^*(\k)$ associated with the map $Z$ are respectively defined by
	$$\z(\k)=Z(z_{k_1}\dots z_{k_n}),~\z^*(\k)=Z^*(z_{k_1}\dots z_{k_n}).$$
If $\k$ is an empty index, then we define
	$$\z(\k)=\z^*(\k)=Z(1)=1.$$

	The $\mathbb{Q}$-linear map $Z:\hh^0 \rightarrow \mathbb{R}$ is an algebra homomorphism with respect to the product $\SH$  i.e for any $w,v \in \hh^0$, $Z$ satisfies the following
	$$Z(w \shuffle v)=Z(w)Z(v).$$
	The above relation is called (finite) double shuffle relation.
	
	We further need to define two functions which are essential for our proof of Theorem 1.2. For any nonnegative integer $m$, the  $\mathbb{Q}$-linear map $\sigma_m:\hh^1 \rightarrow \hh^1$ is defined by $\sigma_m(1)=1$, and
	$$\sigma_m(z_{k_1} \dots z_{k_n})=\sum_{ \substack{a_1+a_2+ \dots +a_n=m \\a_i \geq 0}} z_{k_1+a_1} \dots z_{k_n+a_n}, (\forall n,k_1, \dots , k_n \in \mathbb{N}).$$
	
	Let $\tau$ denotes an antiautomorphism of the noncommutative algebra $\hh$ determined by
	$$\tau(x)=y, \hspace{2em} \tau(y)=x.$$
	Then $\tau$ is an automorphism of the shuffle algebra $\hh_{\SH}$ and also $\tau(\k)=\k'$, where $\k$ and $\k'$ are dual indices.

	

\vspace{1em}	
\section{Proof of Theorem \ref{T2} }	

	We first recall the formula of the generating function of multiple zeta values of height one given by Aomoto (\cite{102}) and Drinfeld(\cite{103}),
		\begin{align}
			\sum_{k,n \geq 1}\zeta(k+1,\underbrace{1, \dots ,1}_{n-1})x^ky^n&=1-\dfrac{\Gamma(1-x)\Gamma(1-y)}{\Gamma(1-x-y)},
		\end{align}
		and the Taylor's expansion of the gamma function
		\begin{align}
			\Gamma(1+x)=exp\bigg(-\gamma x+\sum_{n=2}^{\infty}{(-1)}^n\dfrac{\zeta(n)}{n}x^n\bigg) \hspace{1cm}(|x|<1, \gamma: \text{Euler's constant}).
		\end{align}
		By (1) and (2), it follows that all multiple zeta values of height one can be written as polynomials over $\mathbb{Q}$ in the Riemann zeta values.
On account of this observation, to prove Theorem 1.2, it is sufficient to deduce the following:
\begin{theorem} For any integers $k,n \geq 1$, we have
		\label{T105}
		\begin{align*}
			&{(-1)}^{k}\zeta^*(k+1,\underbrace{1, \dots ,1}_{n})-{(-1)}^{n}\zeta^*(n+1,\underbrace{1, \dots ,1}_{k})\\
			&=k\zeta({k+2,\underbrace{1, \dots, 1}_{n-1}})-n\zeta({n+2,\underbrace{1, \dots, 1}_{k-1}}) \\
			& \ \ \ \ + {(-1)}^{k}\sum_{r=0}^{k-2}{(-1)}^r\zeta({k-r})\zeta({n+1,\underbrace{1, \dots, 1}_{r}}) \\
			& \ \ \ \ - {(-1)}^{n}\sum_{r=0}^{n-2}{(-1)}^r\zeta({n-r})\zeta({k+1,\underbrace{1, \dots, 1}_{r}}).
		\end{align*}
	\end{theorem}

We begin with the following lemma.	
	
	 	\begin{lemma}
		\label{lem1}
		For any integers $k,n \geq 1$,
		\begin{align}\label{3}S(z_{k+1}z_1^n)=\sum_{t=1}^{n+1}\sum_{ \substack{a_1+a_2+ \dots +a_t=n+1-t \\a_i \geq 0, i=1, 2, \dots, t}} z_{a_t+k+1}z_{a_1+1}z_{a_2+1} \dots z_{a_{t-1}+1}.\end{align}
	\end{lemma}
	
	\begin{proof}
		By the definition of the $\mathbb{Q}$-linear map $S$, we have
		
		\begin{align*}
			S(z_{k+1}z_1^n)&=S(x^ky^{n+1}) \nonumber \\  &=\sigma(x^ky^n)y \nonumber \\ &=x^k{(x+y)}^ny \nonumber \\ &=x^k(\sum_{t=1}^{n+1} (x^{n+1-t} \shuffle y^{t-1}))y \nonumber \\
			&=\sum_{t=1}^{n+1} x^k \big(\sum_{ \substack{a_1+a_2+ \dots +a_t=n+1-t \\a_i \geq 0, i=1, 2, \dots, t}} x^{a_t}y x^{a_1}yx^{a_2}y \dots x^{a_{t-1}}\big)y \nonumber \\
			&=\sum_{t=1}^{n+1}\sum_{ \substack{a_1+a_2+ \dots +a_t=n+1-t \\a_i \geq 0, i=1, 2, \dots, t}}x^{a_t+k}y x^{a_1}yx^{a_2}y \dots x^{a_{t-1}}y \nonumber \\
			&=\sum_{t=1}^{n+1}\sum_{ \substack{a_1+a_2+ \dots +a_t=n+1-t \\a_i \geq 0, i=1, 2, \dots, t}} z_{a_t+k+1}z_{a_1+1}z_{a_2+1} \dots z_{a_{t-1}+1}.
		\end{align*}

	\end{proof}
	
	Applying $\mathbb{Q}$-linear map $Z$ on both sides of (\ref{3}), we get the following proposition.
	
	\begin{proposition}
		\label{prop1}
		For any integers $k,n \geq 1$,
		$$\zeta^*(k+1,\underbrace{1, \dots ,1}_{n})=\sum_{t=1}^{n+1}\sum_{ \substack{a_1+a_2+ \dots +a_t=n+1-t \\a_i \geq 0, i=1, 2, \dots, t}} \zeta(a_t+k+1,a_1+1,a_2+1, \dots , a_{t-1}+1).$$
	\end{proposition}

Next we have the following lemma.	
	
	\begin{lemma}\label{lem2}
		For any integers $k,n \geq 1$,
		\begin{align}\label{4}
			\sum_{ \substack{a_1+a_2+ \dots +a_k=n \\a_i \geq 0, i=1, 2, \dots, k}} (a_k+1) z_{a_k+2}z_{a_1+1}z_{a_2+1} \dots z_{a_{k-1}+1} 		
			&=\sum_{t=1}^{n+1}\sum_{ \substack{a_1+a_2+ \dots +a_k=n+1-t \\a_i \geq 0, i=1, 2, \dots, k}} z_{a_k+t+1}z_{a_1+1}z_{a_2+1} \dots z_{a_{k-1}+1}.
		\end{align}
	\end{lemma}
	
	\begin{proof}
		We have,
		\begin{align*}
			L.H.S&=\sum_{ \substack{a_1+a_2+ \dots +a_k=n \\a_i \geq 0, i=1, 2, \dots, k}} (a_k+1) z_{a_k+2}z_{a_1+1}z_{a_2+1} \dots z_{a_{k-1}+1}
		\end{align*}	
		Putting $a_k=0, 1, \dots , n$, we get
		\begin{align*}
			\label{2}
			L.H.S&=z_2\sum_{ \substack{a_1+ \dots +a_{k-1}=n \\a_i \geq 0, i=1, 2, \dots, k}}z_{a_1+1} \dots z_{a_{k-1}+1}+2z_3\sum_{ \substack{a_1+ \dots +a_{k-1}=n-1 \\a_i \geq 0, i=1, 2, \dots, k}}z_{a_1+1} \dots z_{a_{k-1}+1} \nonumber \\
			& \ \ \ \ + 3z_4\sum_{ \substack{a_1+ \dots +a_{k-1}=n-2 \\a_i \geq 0, i=1, 2, \dots, k}}z_{a_1+1} \dots z_{a_{k-1}+1}	+ \dots +nz_{n+1}\sum_{ \substack{a_1+ \dots +a_{k-1}=1 \\a_i \geq 0, i=1, 2, \dots, k}}z_{a_1+1} \dots z_{a_{k-1}+1} \nonumber \\
			& \ \ \ \ +(n+1)z_{n+2}\sum_{ \substack{a_1 \dots +a_{k-1}=0 \\a_i \geq 0, i=1, 2, \dots, k}}z_{a_1+1} \dots z_{a_{k-1}+1} \nonumber \\
			&=\sum_{ \substack{a_1+ \dots +a_{k}=n \\a_i \geq 0, i=1, 2, \dots, k}}z_{a_k+2}z_{a_1+1} \dots z_{a_{k-1}+1}+\sum_{ \substack{a_1+ \dots +a_{k}=n-1 \\a_i \geq 0, i=1, 2, \dots, k}}z_{a_k+3}z_{a_1+1} \dots z_{a_{k-1}+1} \nonumber \\
			& \ \ \ \ + \sum_{ \substack{a_1+ \dots +a_{k}=n-2 \\a_i \geq 0, i=1, 2, \dots, k}}z_{a_k+4}z_{a_1+1} \dots z_{a_{k-1}+1}	+ \dots +\sum_{ \substack{a_1+\dots +a_{k}=1 \\a_i \geq 0, i=1, 2, \dots, k}}z_{a_k+n+1}z_{a_1+1} \dots z_{a_{k-1}+1} \nonumber \\
			& \ \ \ \ +\sum_{ \substack{a_1+ \dots +a_{k}=0 \\a_i \geq 0, i=1, 2, \dots, k}}z_{a_k+n+2}z_{a_1+1} \dots z_{a_{k-1}+1} \nonumber \\
			&=\sum_{t=1}^{n+1}\sum_{ \substack{a_1+ \dots +a_k=n+1-t \\a_i \geq 0, i=1, 2, \dots, k}} z_{a_k+t+1}z_{a_1+1} \dots z_{a_{k-1}+1} \nonumber\\
			&=R.H.S
		\end{align*}	
		Hence the lemma.	
		
	\end{proof}

	\begin{proposition}
		\label{p}
		For any integers $k,n \geq 1$, we have
		\begin{align}
			\zeta^*(k+1,\underbrace{1, \dots ,1}_{n})&=\sum_{t=1}^{n+1}\sum_{ \substack{a_1+a_2+ \dots +a_k=n+1-t \\a_i \geq 0, i=1, 2, \dots, k}} \zeta(a_k+t+1,a_1+1,a_2+1, \dots , a_{k-1}+1)\nonumber\\
			&=\sum_{ \substack{a_1+a_2+ \dots +a_k=n \\a_i \geq 0, i=1, 2, \dots, k}} (a_k+1) \zeta(a_k+2,a_1+1,a_2+1, \dots ,a_{k-1}+1).
		\end{align}
	\end{proposition}	
	
	\begin{proof}
		From lemma \eqref{lem1}, we have
		\begin{align*}
			S(z_{k+1}z_1^n)&=\sum_{t=1}^{n+1}\sum_{ \substack{a_1+a_2+ \dots +a_t=n+1-t \\a_i \geq 0, i=1, 2, \dots, t}} z_{a_t+k+1}z_{a_1+1}z_{a_2+1} \dots z_{a_{t-1}+1}\\
			&=\sum_{t=1}^{n+1} \sigma_{n+1-t}(z_{k+1}z_1^{t-1}).
		\end{align*}
		
		Now
		\begin{align*}
			\sigma_{n+1-t}(\tau ({z_{k+1}z_1^{t-1}}))&=\sigma_{n+1-t}(z_{t+1}z_1^{k-1})\\
			&=\sum_{t=1}^{n+1}\sum_{ \substack{a_1+a_2+ \dots +a_k=n+1-t \\a_i \geq 0, i=1, 2, \dots, k}} z_{a_k+t+1}z_{a_1+1}z_{a_2+1} \dots z_{a_{k-1}+1}.
		\end{align*}
		
		Using Theorem 1.3,
		\begin{align*}
			& \hspace{3em}Z(\sigma_{n+1-t}(z_{k+1}z_1^{t-1}))=Z(\sigma_{n+1-t}(\tau ({z_{k+1}z_1^{t-1}})))\nonumber\\
			& \implies Z(S(z_{k+1}z_1^n))=Z\big(\sum_{t=1}^{n+1}\sum_{ \substack{a_1+a_2+ \dots +a_k=n+1-t \\a_i \geq 0, i=1, 2, \dots, k}} z_{a_k+t+1}z_{a_1+1}z_{a_2+1} \dots z_{a_{k-1}+1}\big)\nonumber\\
			& \implies\zeta^*(k+1,\underbrace{1, \dots ,1}_{n})=\sum_{t=1}^{n+1}\sum_{ \substack{a_1+a_2+ \dots +a_k=n+1-t \\a_i \geq 0, i=1, 2, \dots, k}} \zeta(a_k+t+1,a_1+1,a_2+1, \dots , a_{k-1}+1).
		\end{align*}
		The second equality can be easily obtained by applying the $\mathbb{Q}$-linear map $Z$ on both sides of (4).

	\end{proof}
	

	

	The following is the predominant result for our proof of Theorem 2.1 and hence of Theorem 1.2.
	\begin{lemma}
		\label{lemma4}	
		For any integers $k,n \geq 1$,
		\begin{align*}
			&\sum_{r=0}^{k-2}{(-1)}^r z_{k-r}\shuffle z_{r+2}z_1^{n-1} \nonumber \\
			&=\sum_{i=3}^{n+1} (n+2-i)\sum_{s=2}^{k}z_s \bigg[\sum_{m=2}^{k-s+2}\bigg\{\sum_{p=2}^{k-s-m+3}\sum_{j=0}^{i-4}z_{k-p-s-m+4}z_1^jz_{p}z_1^{i-4-j}+z_{k-s-m+3}z_1^{i-3}\\
			& \ \ \ \ \ \  +\sum_{p=s+m-1}^{k-1}z_{k-p}\sum_{ \substack{\a_j \geq 1;\a_j \neq p-m-s+4; \\ \a_3+\dots +\a_{i-1}=p+i-m-s}} z_{\a_3} \dots z_{\a_{i-1}}\bigg\}z_m\bigg]z_1^{n+1-i} \nonumber \\
			& \ \ \ \ +\{n+{(-1)}^k\}\sum_{\a_1=2}^{k}z_{\a_1}z_{k+2-\a_1}z_1^{n-1}
			+\{1+{(-1)}^k\}(n+1)z_{k+1} z_1^{n}.
 		\end{align*}
	\end{lemma}
	
	\begin{proof}
		For $k=1$ it is obvious. For $k>1$, consider the formula,		
		$$z_l \shuffle z_{k_1}z_{k_2} \dots z_{k_{n-1}}
		=\sum_{i=1}^{n-1}\sum_{ \substack{\a_j \geq 1;\alpha_1+\alpha_2+ \dots +\alpha_{i+1}\\=l+k_1+k_2+ \dots +k_i }} \prod_{j=1}^{i-1} \binom{\alpha_j-1}{k_j-1} \binom{\alpha_i-1}{k_i-\alpha_{i+1}} z_{\alpha_1}z_{\alpha_2} \dots z_{\alpha_{i+1}}z_{k_{i+1}} \dots z_{k_{n-1}}$$~
		$$+\sum_{ \substack{a_j \geq 1;\alpha_1+\alpha_2+ \dots +\alpha_{n}\\=l+k_1+k_2+ \dots +k_{n-1} }} \prod_{j=1}^{n-1} \binom{\alpha_j-1}{k_j-1}z_{\alpha_1}z_{\alpha_2} \dots z_{\alpha_{n}},$$
		which is given in \cite{1}.
		Then, we have
		
		\begin{align}
			\label{equn1}   	
			z_{k-r} \shuffle z_{r+2}z_1^{n-1}&=\sum_{i=1}^{n}\sum_{ \substack{\a_j \geq 1;\alpha_1+\alpha_2+ \dots +\alpha_{i+1}\\=k-r+r+2+\underbrace{1+ \dots+ 1}_{i-1} }} \prod_{j=1}^{i-1} \binom{\alpha_j-1}{k_j-1} \binom{\alpha_i-1}{k_i-\alpha_{i+1}} z_{\alpha_1}z_{\alpha_2} \dots z_{\alpha_{i+1}}z_{k_{i+1}} \dots z_{k_{n}} \nonumber \\
			& \ \ \ \  +\sum_{ \substack{\a_j \geq 1;\alpha_1+\alpha_2+ \dots +\alpha_{n+1}\\=k-r+r+2+n-1}} \prod_{j=1}^{n} \binom{\alpha_j-1}{k_j-1}z_{\alpha_1}z_{\alpha_2} \dots z_{\alpha_{n+1}} \nonumber \\
			& \ \ \ \ \hspace*{4cm} [\text{where}~k_1=r+2; k_2= \dots =k_n=1] \nonumber \\
			&=\sum_{i=1}^{n}\sum_{ \substack{\a_j \geq 1;\alpha_1+\alpha_2+ \dots +\alpha_{i+1}\\=k+i+1 }} \prod_{j=1}^{i-1} \binom{\alpha_j-1}{k_j-1} \binom{\alpha_i-1}{k_i-\alpha_{i+1}} z_{\alpha_1}z_{\alpha_2} \dots z_{\alpha_{i+1}}z_{k_{i+1}} \dots z_{k_{n}} \nonumber \\
			& \ \ \ \ +\sum_{ \substack{\a_j \geq 1;\alpha_1+\alpha_2+ \dots +\alpha_{n+1}\\=k+n+1}} \prod_{j=1}^{n} \binom{\alpha_j-1}{k_j-1}z_{\alpha_1}z_{\alpha_2} \dots z_{\alpha_{n+1}} \nonumber \\
			&=\sum_{\a_1+\a_2=k+2} \binom{\a_1-1}{r+2-\a_2} z_{\a_1}z_{\a_2}z_1^{n-1} \nonumber \\
			& \ \ \ \ +\sum_{i=2}^{n}\sum_{ \substack{\a_j \geq 1;\alpha_1+\alpha_2+ \dots +\alpha_{i+1}\\=k+i+1 }}  \binom{\alpha_1-1}{r+1} \binom{\alpha_i-1}{1-\alpha_{i+1}} z_{\alpha_1}z_{\alpha_2} \dots z_{\alpha_{i+1}}z_1^{n-i} \nonumber \\
			& \ \ \ \ +\sum_{ \substack{\a_j \geq 1;\alpha_1+\alpha_2+ \dots +\alpha_{n+1}\\=k+n+1}}  \binom{\alpha_1-1}{r+1}z_{\alpha_1}z_{\alpha_2} \dots z_{\alpha_{n+1}}\nonumber \\
			&=\sum_{\a_1+\a_2=k+2} \binom{\a_1-1}{r+2-\a_2} z_{\a_1}z_{\a_2}z_1^{n-1} \nonumber \\
			& \ \ \ \ +\sum_{i=2}^{n}\sum_{ \substack{\a_j \geq 1;\alpha_1+\alpha_2+ \dots +\alpha_{i}\\=k+i}}  \binom{\alpha_1-1}{r+1} z_{\alpha_1}z_{\alpha_2} \dots z_{\alpha_{i}}z_1^{n-i+1} \nonumber \\
			& \ \ \ \ +\sum_{ \substack{\a_j \geq 1;\alpha_1+\alpha_2+ \dots +\alpha_{n+1}\\=k+n+1}}  \binom{\alpha_1-1}{r+1}z_{\alpha_1}z_{\alpha_2} \dots z_{\alpha_{n+1}} \nonumber \\
			&=\sum_{\a_1+\a_2=k+2} \binom{\a_1-1}{r+2-\a_2} z_{\a_1}z_{\a_2}z_1^{n-1} \nonumber\\
			& \ \ \ \ +\sum_{i=2}^{n+1}\sum_{ \substack{\a_j \geq 1;\alpha_1+\alpha_2+ \dots +\alpha_{i}\\=k+i}}  \binom{\alpha_1-1}{r+1} z_{\alpha_1}z_{\alpha_2} \dots z_{\alpha_{i}}z_1^{n-i+1}.
		\end{align}
		
		In the first summation of \eqref{equn1}, the minimum value of $\a_2$ is 1, and the maximum value of $\a_2$ is $r+2$, so the minimum value of $\a_1$ is $k-r$, and the maximum value of $\a_1$ is $k+1$. In the second summation, the minimum value of $\a_1$ is $r+2$, and the maximum value of $\a_1$ is $k+1$.
		So \eqref{equn1} becomes

		\begin{align}
			\label{equn2}
			z_{k-r} \shuffle z_{r+2}z_1^{n-1}&=\sum_{\a_1=k-r}^{k+1} \binom{\a_1-1}{\a_1-{(k-r)}} z_{\a_1}z_{k+2-\a_1}z_1^{n-1} \nonumber \\
			& \ \ \ \ +\sum_{i=2}^{n+1} \sum_{\a_1=r+2}^{k+1} \binom{\a_1-1}{r+1}\sum_{ \substack{\a_j \geq 1;\alpha_2+\alpha_3+ \dots +\alpha_{i}\\=k+i-\a_1}} z_{\alpha_1}z_{\alpha_2} \dots z_{\alpha_{i}}z_1^{n-i+1} \nonumber \\
			&=\sum_{\a_1=k-r}^{k+1} \binom{\a_1-1}{k-r-1} z_{\a_1}z_{k+2-\a_1}z_1^{n-1}+\sum_{\a_1=r+2}^{k+1} \binom{\a_1-1}{r+1} z_{\a_1}z_{k+2-\a_1}z_1^{n-1} \nonumber \\
			& \ \ \ \ +\sum_{i=3}^{n+1} \sum_{\a_1=r+2}^{k+1} \binom{\a_1-1}{r+1}\sum_{ \substack{\a_j \geq 1;\alpha_2+\alpha_3+ \dots +\alpha_{i}\\=k+i-\a_1}} z_{\alpha_1}z_{\alpha_2} \dots z_{\alpha_{i}}z_1^{n-i+1} \nonumber\\
			&=A+B+C .
		\end{align}
		where A, B, and C are the first, second, and third sum, respectively, of \eqref{equn2}.
		
		We have
		\begin{align}
			\label{equn3}
			C&=\sum_{i=3}^{n+1} \sum_{\a_1=r+2}^{k+1} \binom{\a_1-1}{r+1}\sum_{ \substack{\a_j \geq 1;\alpha_2+\alpha_3+ \dots +\alpha_{i}\\=k+i-\a_1}} z_{\alpha_1}z_{\alpha_2} \dots z_{\alpha_{i}}z_1^{n-i+1} \nonumber \\
			&=\sum_{\a_1=r+2}^{k+1} \binom{\a_1-1}{r+1} z_{\alpha_1}\sum_{i=3}^{n+1}\sum_{ \substack{\a_j \geq 1;\alpha_2+\alpha_3+ \dots +\alpha_{i}\\=k+i-\a_1}} z_{\alpha_2}z_{\alpha_3} \dots z_{\alpha_{i}}z_1^{n-i+1}.
		\end{align}

		Now

		\begin{align}
			\label{equn4}
			&\sum_{i=3}^{n+1}\sum_{ \substack{\a_j \geq 1;\a_2+\a_3+ \dots +\a_{i}\\=k+i-\a_1}} z_{\a_2}z_{\a_3} \dots z_{\a_{i}}z_1^{n-i+1} \nonumber \\&=\sum_{ \substack{\a_j \geq 1;\a_2+\a_3=k+3-\a_1}}z_{\a_2}z_{\a_{3}}z_1^{n-2}+\sum_{ \substack{\a_j \geq 1;\a_2+\a_3+\a_4=k+4-\a_1}}z_{\a_2}z_{\a_{3}}z_{\a_4}z_1^{n-3} \nonumber \\
			& \ \ \ \ + \dots + \sum_{ \substack{\a_j \geq 1;\a_2+\a_3+ \dots +\a_{n}\\=k+n-\a_1}} z_{\a_2} \dots z_{\a_{n}}z_1 \nonumber \\
			& \ \ \ \ +\sum_{ \substack{\a_j \geq 1;\a_2+\a_3+ \dots +\a_{n+1}\\=k+n+1-\a_1}} z_{\a_2} \dots z_{\a_{n+1}}.
		\end{align}

		Separating from each sum on R.H.S of \eqref{equn4}, the term where each of $\a_4, \a_5, \dots , \a_{n+1}$ is 1 , RHS of \eqref{equn4} becomes

		\begin{align}
			\label{equn5}
			&(n-1)\sum_{ \substack{\a_j \geq 1;\a_2+\a_3=k+3-\a_1}}z_{\a_2}z_{\a_{3}}z_1^{n-2}+ \sum_{ \substack{\a_2+\a_3+\a_4=k+4-\a_1\\\a_j \geq 1;\a_4 \geq 2}}z_{\a_2}z_{\a_{3}}z_{\a_4}z_1^{n-3}+\nonumber\\
&\sum_{ \substack{\a_2+\a_3+\a_4+\a_5=k+5-\a_1\\\a_j \geq 1;\a_2+\a_{3} \neq k+3-\a_1 }}z_{\a_2}z_{\a_{3}}z_{\a_4}z_{\a_5}z_1^{n-4} + \dots +\sum_{ \substack{\a_j \geq 1;\a_2+\a_3+ \dots +\a_{n}\\=k+n-\a_1;\a_2+\a_{3} \neq k+3-\a_1}} z_{\a_2} \dots z_{\a_{n}}z_1+\nonumber\\
&\sum_{ \substack{\a_j \geq 1;\a_2+\a_3+ \dots +\a_{n+1}\\=k+n+1-\a_1;\a_2+\a_{3} \neq k+3-\a_1}} z_{\a_2} \dots z_{\a_{n+1}}.
		\end{align}

		In a similar manner, we successively separate the terms with each of $\a_j,\a_{j+1}, \dots, \a_{n+1}=1$ for $j=5,6,7, \dots, n+1$ from all the sums of \eqref{equn5}, as a result of which we get

		\begin{align}
			\label{equn7}
			R.H.S\text{ of } \eqref{equn4} &=(n-1)\sum_{ \substack{\a_j \geq 1;\a_2+\a_3=k+3-\a_1}}z_{\a_2}z_{\a_{3}}z_1^{n-2} +(n-2)\sum_{ \substack{\a_2+\a_3+\a_4=k+4-\a_1\\ \a_j \geq 1;\a_4 \geq 2}} z_{\a_2}z_{\a_{3}}z_{\a_4}z_1^{n-3} \nonumber \\
			& \ \ \ \ +(n-3)\sum_{ \substack{\a_2+\a_3+\a_4+\a_5=k+5-\a_1\\ \a_j \geq 1;\a_5 \geq 2}}z_{\a_2}z_{\a_{3}}z_{\a_4}z_{\a_5}z_1^{n-4} + \dots + \nonumber \\
			&  \ \ \ \ + 2\sum_{ \substack{\a_j \geq 1;\a_2+\a_3+ \dots +\a_{n}\\=k+n-\a_1;\a_n \geq 2}} z_{\a_2} \dots z_{\a_{n}}z_1+\sum_{ \substack{\a_j \geq 1;\a_2+\a_3+ \dots +\a_{n+1}\\=k+n+1-\a_1;\a_{n+1} \geq 2}} z_{\a_2} \dots z_{\a_{n+1}} \nonumber \\
			&=(n-1)\sum_{ \substack{\a_j \geq 1;\a_2=k+2-\a_1}}z_{\a_2}z_1^{n-1} + (n-1)\sum_{ \substack{\a_2+\a_3=k+3-\a_1\\ \a_j \geq 1; \a_3 \geq 2}}z_{\a_2}z_{\a_{3}}z_1^{n-2} \nonumber \\
			& \ \ \ \ +(n-2)\sum_{ \substack{\a_2+\a_3+\a_4=k+4-\a_1\\ \a_j \geq 1;\a_4 \geq 2}} z_{\a_2}z_{\a_{3}}z_{\a_4}z_1^{n-3} \nonumber \\
			& \ \ \ \ +(n-3)\sum_{ \substack{\a_2+\a_3+\a_4+\a_5=k+5-\a_1\\ \a_j \geq 1;\a_5 \geq 2}}z_{\a_2}z_{\a_{3}}z_{\a_4}z_{\a_5}z_1^{n-4} + \dots +  \nonumber \\
			&  \ \ \ \ + 2\sum_{ \substack{\a_j \geq 1;\a_2+\a_3+ \dots +\a_{n}\\=k+n-\a_1;\a_n \geq 2}} z_{\a_2} \dots z_{\a_{n}}z_1+\sum_{ \substack{\a_j \geq 1;\a_2+\a_3+ \dots +\a_{n+1}\\=k+n+1-\a_1;\a_{n+1} \geq 2}} z_{\a_2} \dots z_{\a_{n+1}} \nonumber \\
			&=(n-1)z_{k+2-\a_1}z_1^{n-1} + \sum_{i=3}^{n+1} (n+2-i)\sum_{ \substack{\a_j \geq 1;\a_2+\a_3+ \dots +\a_{i}\\=k+i-\a_1;\a_{i} \geq 2}} z_{\a_2} \dots z_{\a_{i}}z_1^{n+1-i}.
		\end{align}

		Therefore, from \eqref{equn4} and \eqref{equn7},
		
		\begin{align}
			\label{equn8}
			&\sum_{i=3}^{n+1}\sum_{ \substack{\a_j \geq 1;\a_2+\a_3+ \dots +\a_{i}\\=k+i-\a_1}} z_{\a_2}z_{\a_3} \dots z_{\a_{i}}z_1^{n-i+1} \nonumber \\
			&=(n-1)z_{k+2-\a_1}z_1^{n-1} + \sum_{i=3}^{n+1} (n+2-i)\sum_{ \substack{\a_j \geq 1;\a_2+\a_3+ \dots +\a_{i}\\=k+i-\a_1;\a_{i} \geq 2}} z_{\a_2} \dots z_{\a_{i}}z_1^{n+1-i}.
		\end{align}
		
		Hence from \eqref{equn3} and \eqref{equn8},

		\begin{align}
			\label{equn9}
			C&=\sum_{\a_1=r+2}^{k+1} \binom{\a_1-1}{r+1} z_{\a_1}\bigg[(n-1)z_{k+2-\a_1}z_1^{n-1} + \sum_{i=3}^{n+1} (n+2-i) \nonumber \\
			& \ \ \ \hspace*{5cm} \sum_{ \substack{\a_j \geq 1;\a_2+\a_3+ \dots +\a_{i}\\=k+i-\a_1;\a_{i} \geq 2}} z_{\a_2} \dots z_{\a_{i}}z_1^{n+1-i}\bigg] \nonumber \\
			&=(n-1)\sum_{\a_1=r+2}^{k+1}\binom{\a_1-1}{r+1}z_{\a_1}z_{k+2-\a_1}z_1^{n-1}\nonumber \\
			& \ \ \ \  + \sum_{\a_1=r+2}^{k}\binom{\a_1-1}{r+1}z_{\a_1}\sum_{i=3}^{n+1} (n+2-i)\sum_{ \substack{\a_j \geq 1;\a_2+\a_3+ \dots +\a_{i}\\=k+i-\a_1;\a_{i} \geq 2}} z_{\a_2} \dots z_{\a_{i}}z_1^{n+1-i}.
		\end{align}

		From \eqref{equn2},

		\begin{align}
			\label{equn10}
			&z_{k-r} \shuffle z_{r+2}z_1^{n-1} \nonumber \\
			&=\sum_{\a_1=k-r}^{k+1} \binom{\a_1-1}{k-r-1} z_{\a_1}z_{k+2-\a_1}z_1^{n-1}+\sum_{\a_1=r+2}^{k+1} \binom{\a_1-1}{r+1} z_{\a_1}z_{k+2-\a_1}z_1^{n-1} \nonumber \\
			 & \ \ \ \  +(n-1)\sum_{\a_1=r+2}^{k+1}\binom{\a_1-1}{r+1}z_{\a_1}z_{k+2-\a_1}z_1^{n-1}\nonumber \\
			 & \ \ \ \  + \sum_{\a_1=r+2}^{k}\binom{\a_1-1}{r+1}z_{\a_1}\sum_{i=3}^{n+1} (n+2-i)\sum_{ \substack{\a_j \geq 1;\a_2+\a_3+ \dots +\a_{i}\\=k+i-\a_1;\a_{i} \geq 2}} z_{\a_2} \dots z_{\a_{i}}z_1^{n+1-i}\nonumber \\
			&=\sum_{\a_1=k-r}^{k+1} \binom{\a_1-1}{k-r-1} z_{\a_1}z_{k+2-\a_1}z_1^{n-1}+ n\sum_{\a_1=r+2}^{k+1} \binom{\a_1-1}{r+1} z_{\a_1}z_{k+2-\a_1}z_1^{n-1} \nonumber \\
			& \ \ \ \ + \sum_{\a_1=r+2}^{k} \binom{\a_1-1}{r+1}z_{\a_1}\sum_{i=3}^{n+1} (n+2-i)\sum_{ \substack{\a_j \geq 1;\a_2+\a_3+ \dots +\a_{i}\\=k+i-\a_1;\a_{i} \geq 2}} z_{\a_2} \dots z_{\a_{i}}z_1^{n+1-i}.
		\end{align}

		Therefore,
		\begin{align}
			\label{equn11}   	
			&\sum_{r=0}^{k-2}{(-1)}^r z_{k-r} \shuffle z_{r+2}z_1^{n-1}\nonumber\\
			&=\sum_{r=0}^{k-2}{(-1)}^r\bigg[ \sum_{\a_1=k-r}^{k+1} \binom{\a_1-1}{k-r-1} z_{\a_1}z_{k+2-\a_1}z_1^{n-1} + n\sum_{\a_1=r+2}^{k+1} \binom{\a_1-1}{r+1} z_{\a_1}z_{k+2-\a_1}z_1^{n-1} \nonumber \\
			& \ \ \ \ + \sum_{\a_1=r+2}^{k} \binom{\a_1-1}{r+1}z_{\a_1}\sum_{i=3}^{n+1} (n+2-i)  \sum_{ \substack{\a_j \geq 1;\a_2+\a_3+ \dots +\a_{i}\\=k+i-\a_1;\a_{i} \geq 2}} z_{\a_2} \dots z_{\a_{i}}z_1^{n+1-i}\bigg].
		\end{align}

		First sum on the RHS of \eqref{equn11} is equal to
		\begin{align}
			\label{equn12}
			&\sum_{r=0}^{k-2}{(-1)}^r\sum_{\a_1=k-r}^{k+1} \binom{\a_1-1}{k-r-1} z_{\a_1}z_{k+2-\a_1}z_1^{n-1} \nonumber \\
			&=\sum_{\a_1=k}^{k+1} \binom{\a_1-1}{k-1} z_{\a_1}z_{k+2-\a_1}z_1^{n-1} -\sum_{\a_1=k-1}^{k+1} \binom{\a_1-1}{k-2} z_{\a_1}z_{k+2-\a_1}z_1^{n-1} \nonumber \\
			& \ \ \ \hspace{3em} + \dots +{(-1)}^{k-3}\sum_{\a_1=3}^{k+1} \binom{\a_1-1}{2} z_{\a_1}z_{k+2-\a_1}z_1^{n-1} +{(-1)}^{k-2}\sum_{\a_1=2}^{k+1} \binom{\a_1-1}{1} z_{\a_1}z_{k+2-\a_1}z_1^{n-1} \nonumber \\
			&={(-1)}^{k-2}\bigg[\sum_{\a_1=2}^{k+1} \binom{\a_1-1}{1} z_{\a_1}z_{k+2-\a_1}z_1^{n-1} -\sum_{\a_1=3}^{k+1} \binom{\a_1-1}{2} z_{\a_1}z_{k+2-\a_1}z_1^{n-1}+ \dots  \nonumber \\
			& \ \ \ \hspace{3em} +{(-1)}^{k-3}\sum_{\a_1=k-1}^{k+1} \binom{\a_1-1}{k-2} z_{\a_1}z_{k+2-\a_1}z_1^{n-1} + {(-1)}^{k-2}\sum_{\a_1=k}^{k+1} \binom{\a_1-1}{k-1} z_{\a_1}z_{k+2-\a_1}z_1^{n-1}\bigg] \nonumber \\
			&={(-1)}^{k}\sum_{r=0}^{k-2}{(-1)}^r\sum_{\a_1=r+2}^{k+1} \binom{\a_1-1}{r+1} z_{\a_1}z_{k+2-\a_1}z_1^{n-1}.
		\end{align}

		Hence from \eqref{equn11} and \eqref{equn12}, we have

		\begin{align*}
			&\sum_{r=0}^{k-2}{(-1)}^r z_{k-r}\shuffle z_{r+2}z_1^{n-1}\\
			&=\sum_{r=0}^{k-2}{(-1)}^r\bigg[ \{n+{(-1)}^k\}\sum_{\a_1=r+2}^{k+1} \binom{\a_1-1}{r+1} z_{\a_1}z_{k+2-\a_1}z_1^{n-1}\\
			& \ \ \ \ + \sum_{\a_1=r+2}^{k} \binom{\a_1-1}{r+1}z_{\a_1}\sum_{i=3}^{n+1} (n+2-i)\sum_{ \substack{\a_j \geq 1;\a_2+\a_3+ \dots +\a_{i}\\=k+i-\a_1;\a_{i} \geq 2}} z_{\a_2} \dots z_{\a_{i}}z_1^{n+1-i}\bigg]\\
			&=\sum_{r=0}^{k-2}{(-1)}^r\bigg[\sum_{\a_1=r+2}^{k} \binom{\a_1-1}{r+1} z_{\a_1}\bigg\{\sum_{i=3}^{n+1} (n+2-i)\sum_{ \substack{\a_j \geq 1;\a_2+\a_3+ \dots +\a_{i}\\=k+i-\a_1;\a_{i} \geq 2}} z_{\a_2} \dots z_{\a_{i}}z_1^{n+1-i} \\
			& \ \ \ \ \hspace{10em}+ \{n+{(-1)}^k\}z_{k+2-\a_1}z_1^{n-1}\bigg\} +\{n+{(-1)}^k\}\binom{k}{r+1}z_{k+1}z_1^{n}\bigg]\\
			&=\sum_{\a_1=2}^{k} \binom{\a_1-1}{1} z_{\a_1}\bigg\{\sum_{i=3}^{n+1} (n+2-i)\sum_{ \substack{\a_j \geq 1;\a_2+\a_3+ \dots +\a_{i}\\=k+i-\a_1;\a_{i} \geq 2}} z_{\a_2} \dots z_{\a_{i}}z_1^{n+1-i} + \{n+{(-1)}^k\}z_{k+2-\a_1}z_1^{n-1}\bigg\} \nonumber \\
			& \ \ \ \ - \sum_{\a_1=3}^{k} \binom{\a_1-1}{2} z_{\a_1}\bigg\{\sum_{i=3}^{n+1} (n+2-i)\sum_{ \substack{\a_j \geq 1;\a_2+\a_3+ \dots +\a_{i}\\=k+i-\a_1;\a_{i} \geq 2}} z_{\a_2} \dots z_{\a_{i}}z_1^{n+1-i} + \{n+{(-1)}^k\}z_{k+2-\a_1}z_1^{n-1}\bigg\} \nonumber \\
				\end{align*}
		\begin{align}
			\label{equn14}
			& \ \ \ \ + \dots  +{(-1)}^{k-3}\sum_{\a_1=k-1}^{k} \binom{\a_1-1}{k-2} z_{\a_1}\bigg\{\sum_{i=3}^{n+1} (n+2-i)\sum_{ \substack{\a_j \geq 1;\a_2+\a_3+ \dots +\a_{i}\\=k+i-\a_1;\a_{i} \geq 2}} z_{\a_2} \dots z_{\a_{i}}z_1^{n+1-i} \nonumber \\
			& \ \ \ \ \hspace{20em} + \{n+{(-1)}^k\}z_{k+2-\a_1}z_1^{n-1}\bigg\} \nonumber \\
			& \ \ \ \ + {(-1)}^{k-2}\sum_{\a_1=k}^{k} \binom{\a_1-1}{k-1} z_{\a_1}\bigg\{\sum_{i=3}^{n+1} (n+2-i)\sum_{ \substack{\a_j \geq 1;\a_2+\a_3+ \dots +\a_{i}\\=k+i-\a_1;\a_{i} \geq 2}} z_{\a_2} \dots z_{\a_{i}}z_1^{n+1-i} \nonumber \\
			& \ \ \ \ \hspace{9em} +\{n+{(-1)}^k\}z_{k+2-\a_1}z_1^{n-1}\bigg\}+\{n+{(-1)}^k\}\sum_{r=0}^{k-2}{(-1)}^r\binom{k}{r+1}z_{k+1}z_1^{n} \nonumber \\
			&=\binom{1}{1}z_2\bigg\{\sum_{i=3}^{n+1} (n+2-i)\sum_{ \substack{\a_j \geq 1;\a_2+\a_3+ \dots +\a_{i}\\=k+i-2;\a_{i} \geq 2}} z_{\a_2} \dots z_{\a_{i}}z_1^{n+1-i}+\{n+{(-1)}^k\}z_{k+2-2}z_1^{n-1}\bigg\} \nonumber \\
			& \ \ \ \ + \bigg\{\binom{2}{1}-\binom{2}{2}\bigg\}z_3\bigg\{\sum_{i=3}^{n+1} (n+2-i)\sum_{ \substack{\a_j \geq 1;\a_2+\a_3+ \dots +\a_{i}\\=k+i-3;\a_{i} \geq 2}} z_{\a_2} \dots z_{\a_{i}}z_1^{n+1-i}+\{n+{(-1)}^k\}z_{k+2-3}z_1^{n-1}\bigg\} \nonumber \\
			& \ \ \ \ \hspace{20em}+\{n+{(-1)}^k\}z_{k+2-4}z_1^{n-1}\bigg\} \nonumber \\
			& \ \ \ \ + \dots + \bigg\{\binom{k-1}{1}-\binom{k-1}{2}+\binom{k-1}{3}- \dots + {(-1)}^{k-3}\binom{k-1}{k-2} +{(-1)}^{k-2}\binom{k-1}{k-1}\bigg\} \nonumber \\
			& \ \ \ \ \hspace{8em}z_k\bigg\{\sum_{i=3}^{n+1} (n+2-i)\sum_{ \substack{\a_j \geq 1;\a_2+\a_3+ \dots +\a_{i}\\=i;\a_{i} \geq 2}} z_{\a_2} \dots z_{\a_{i}}z_1^{n+1-i}+\{n+{(-1)}^k\}z_{2}z_1^{n-1}\bigg\} \nonumber \\
			& \ \ \ \ + \{n+{(-1)}^k\}\sum_{r=0}^{k-2}{(-1)}^r\binom{k}{r+1} z_{k+1}z_1^{n}\nonumber \\
			&=\sum_{l=2}^{k}\bigg\{\binom{l-1}{1}-\binom{l-1}{2}+\binom{l-1}{3}- \dots + {(-1)}^{l-3}\binom{l-1}{l-2} +{(-1)}^{l-2}\binom{l-1}{l-1}\bigg\} \nonumber \\
			& \ \ \ \ z_l\bigg\{\sum_{i=3}^{n+1} (n+2-i)\sum_{ \substack{\a_j \geq 1;\a_2+\a_3+ \dots +\a_{i}\\=k+i-l;\a_{i} \geq 2}} z_{\a_2} \dots z_{\a_{i}}z_1^{n+1-i}+\{n+{(-1)}^k\}z_{k+2-l}z_1^{n-1}\bigg\}\nonumber \\
			& \ \ \ \ + \{n+{(-1)}^k\}\sum_{r=0}^{k-2}{(-1)}^r\binom{k}{r+1} z_{k+1}z_1^{n}\nonumber \\
			&=\sum_{l=2}^{k}\bigg\{ \sum_{v=1}^{l-1}{(-1)}^{v-1}\binom{l-1}{v}\bigg\}z_l\bigg\{\sum_{i=3}^{n+1} (n+2-i)\sum_{ \substack{\a_j \geq 1;\a_2+\a_3+ \dots +\a_{i}\\=k+i-l;\a_{i} \geq 2}} z_{\a_2} \dots z_{\a_{i}}z_1^{n+1-i} \nonumber\\
			& \ \ \ \ +\{n+{(-1)}^k\}z_{k+2-l}z_1^{n-1}\bigg\}  + \bigg\{ \sum_{r=0}^{k-2}{(-1)}^{r}\binom{k}{r+1}\bigg\}\{n+{(-1)}^k\}z_{k+1} z_1^{n}.
		\end{align}

		Since for any $k \in \mathbb{N}$,  $$\sum_{r=0}^{k}{(-1)}^{r}\binom{k}{r}=0,$$ 

we have
		\begin{align}
			\label{equn15}
			& \ \ \hspace*{2em}\sum_{r=0}^{k-2}{(-1)}^{r+1}\binom{k}{r+1}+1+{(-1)}^k=0 \nonumber \\
			&\implies  \sum_{r=0}^{k-2}{(-1)}^{r}\binom{k}{r+1}=\{1+{(-1)}^k\}.
		\end{align}
		
		Similarly,
		\begin{align}
			\label{equn16}
			& \ \ \hspace*{2em} \sum_{v=0}^{l-1}{(-1)}^{v}\binom{l-1}{v}=0 \nonumber \\
			&\implies \sum_{v=1}^{l-1}{(-1)}^{v}\binom{l-1}{v}+1=0 \nonumber \\
			&\implies \sum_{v=1}^{l-1}{(-1)}^{v-1}\binom{l-1}{v}-1=0 \nonumber \\
			&\implies \sum_{v=1}^{l-1}{(-1)}^{v-1}\binom{l-1}{v}=1.
		\end{align}

		Now using \eqref{equn15} and \eqref{equn16} in \eqref{equn14}, we get
		
		\begin{align}
			\label{equn19}
			&\sum_{r=0}^{k-2}{(-1)}^r z_{k-r}\shuffle z_{r+2}z_1^{n-1} \nonumber \\
			&=\sum_{\a_1=2}^{k}z_{\a_1}\bigg\{\sum_{i=3}^{n+1} (n+2-i)\sum_{ \substack{\a_j \geq 1;\a_2+\a_3+ \dots +\a_{i}\\=k+i-\a_1;\a_{i} \geq 2}} z_{\a_2} \dots z_{\a_{i}}z_1^{n+1-i}+\{n+{(-1)}^k\}z_{k+2-\a_1}z_1^{n-1}\bigg\} \nonumber \\
			& \ \ \ \ \hspace{8em} + \{1+{(-1)}^k\}\{n+{(-1)}^k\}z_{k+1} z_1^{n} \nonumber \\
			&=\sum_{\a_1=2}^{k}z_{\a_1}\bigg\{\sum_{i=3}^{n+1} (n+2-i)\sum_{ \substack{\a_j \geq 1;\a_2+\a_3+ \dots +\a_{i}\\=k+i-\a_1;\a_{i} \geq 2}} z_{\a_2} \dots z_{\a_{i}}z_1^{n+1-i}\bigg\} \nonumber \\
			& \ \ \ \ +\{n+{(-1)}^k\}\sum_{\a_1=2}^{k}z_{\a_1}z_{k+2-\a_1}z_1^{n-1} +\{1+{(-1)}^k\}(n+1)z_{k+1} z_1^{n}.\nonumber\\
			&=\sum_{i=3}^{n+1} (n+2-i)\sum_{\a_1=2}^{k}z_{\a_1}\sum_{ \substack{\a_j \geq 1;\a_2+\a_3+ \dots +\a_{i}\\=k+i-\a_1;\a_{i} \geq 2}} z_{\a_2} \dots z_{\a_{i}}z_1^{n+1-i} \nonumber \\
			& \ \ \ \ +\{n+{(-1)}^k\}\sum_{\a_1=2}^{k}z_{\a_1}z_{k+2-\a_1}z_1^{n-1}
			+\{1+{(-1)}^k\}(n+1)z_{k+1} z_1^{n}.
		\end{align}
		
		Let $$P=\sum_{\a_1=2}^{k}z_{\a_1}\sum_{ \substack{\a_j \geq 1;\a_2+\a_3+ \dots +\a_{i}\\=k+i-\a_1;\a_{i} \geq 2}} z_{\a_2} \dots z_{\a_{i}}z_1^{n+1-i}.$$
		Then
		\begin{align}
			\label{equn20}
			P=P_2+P_3+  \dots + P_{k-1}+P_k ~,
		\end{align}
		where
		$$P_s=z_{s}\sum_{ \substack{\a_j \geq 1;\a_2+\a_3+ \dots +\a_{i}\\=k+i-s;\a_{i} \geq 2}} z_{\a_2} \dots z_{\a_{i}}z_1^{n+1-i}\text{,  }\text{ }2\leq s\leq k.$$
	
		In $P_s$, minimum value of $\a_{i}$ is 2 and the maximum value of $\a_{i}$ is $k-s+2$.\\ 
		Therefore,
		
		\begin{align*}
			P_s&=z_s\bigg[Q_2+Q_3 + \dots +Q_{k-s+2}\bigg]z_1^{n+1-i},
		\end{align*}
			where $$Q_m=\sum_{ \substack{\a_j \geq 1;\a_2+\a_3+ \dots +\a_{i-1}\\=k+i-s-m}} z_{\a_2} \dots z_{\a_{i-1}}z_m,\text{ }\text{ }2\leq m\leq k-s+2.$$
		In $Q_m$, min value of $\a_2$ is 1, and the max value of $\a_2$ is $k-s-m+3$.\\
		Hence 
		\begin{align*}
			Q_m&=\bigg[\sum_{p=2}^{k-s-m+3}\sum_{j=0}^{i-4}z_{k-p-s-m+4}z_1^jz_{p}z_1^{i-4-j}+z_{k-s-m+3}z_1^{i-3}\\
			& \ \ \ \ \ \  +\sum_{p=s+m-1}^{k-1}z_{k-p}\sum_{ \substack{\a_j \geq 1;\a_j \neq p-m-s+4; \\ \a_3+ \dots +\a_{i-1}=p+i-m-s}} z_{\a_3} \dots z_{\a_{i-1}}\bigg]z_m.
		\end{align*}
		
	Therefore,
	\begin{align*}
		P_s&=z_s \bigg[\sum_{m=2}^{k-s+2}\bigg\{\sum_{p=2}^{k-s-m+3}\sum_{j=0}^{i-4}z_{k-p-s-m+4}z_1^jz_{p}z_1^{i-4-j}+z_{k-s-m+3}z_1^{i-3}\\
		& \ \ \ \ \ \  +\sum_{p=s+m-1}^{k-1}z_{k-p}\sum_{ \substack{\a_j \geq 1;\a_j \neq p-m-s+4; \\ \a_3+ \dots +\a_{i-1}=p+i-m-s}} z_{\a_3} \dots z_{\a_{i-1}}\bigg\}z_m\bigg]z_1^{n+1-i}.
	\end{align*}

Thus 
\begin{align*}
	P&=\sum_{s=2}^{k}z_s \bigg[\sum_{m=2}^{k-s+2}\bigg\{\sum_{p=2}^{k-s-m+3}\sum_{j=0}^{i-4}z_{k-p-s-m+4}z_1^jz_{p}z_1^{i-4-j}+z_{k-s-m+3}z_1^{i-3}\\
	& \ \ \ \ \ \  +\sum_{p=s+m-1}^{k-1}z_{k-p}\sum_{ \substack{\a_j \geq 1;\a_j \neq p-m-s+4; \\ \a_3+\dots +\a_{i-1}=p+i-m-s}} z_{\a_3} \dots z_{\a_{i-1}}\bigg\}z_m\bigg]z_1^{n+1-i}. 
\end{align*}

		
		Putting the expression of P in \eqref{equn19}, we get the lemma.


	\end{proof}

	\begin{theorem}
		For any integers $k, n \geq 1$,
		
		\begin{align}
			\label{equn46}
			\sum_{r=0}^{k-2} {(-1)}^r z_2z_1^{k-r-2} \shuffle z_{n+1}z_1^{r}&= \sum_{ \substack{a_1+a_2+ \dots +a_k=n \\a_i \geq 0}} (a_k+1) z_{a_k+2}z_{a_1+1}z_{a_2+1} \dots z_{a_{k-1}+1} \nonumber \\
			& \ \ \ \  +{(-1)}^k\bigg[z_{n+1}z_2z_1^{k-2}+z_{n+1}z_1z_2z_1^{k-3}\nonumber \\ & \ \ \ \ + \dots + z_{n+1}z_1^{k-2}z_2 + {(n+1)}z_{n+2}z_1^{k-1}\bigg].	
		\end{align}	
	\end{theorem}
	
	\begin{proof}
		For $k=1$ it is obvious. \\
		From lemma \ref{lemma4}, we have
		\begin{align}
			\label{equn47}
			&\sum_{r=0}^{k-2}{(-1)}^r z_{k-r}\shuffle z_{r+2}z_1^{n-1} \nonumber \\
			&=U+V+W,
		\end{align}
		
		where
		\begin{align*}
			&U=\sum_{i=3}^{n+1} (n+2-i)\sum_{s=2}^{k}z_s \bigg[\sum_{m=2}^{k-s+2}\bigg\{\sum_{p=2}^{k-s-m+3}\sum_{j=0}^{i-4}z_{k-p-s-m+4}z_1^jz_{p}z_1^{i-4-j}+z_{k-s-m+3}z_1^{i-3}\\
			& \ \ \ \ \ \  +\sum_{p=s+m-1}^{k-1}z_{k-p}\sum_{ \substack{\a_j \geq 1;\a_j \neq p-m-s+4; \\ \a_3+\dots +\a_{i-1}=p+i-m-s}} z_{\a_3} \dots z_{\a_{i-1}}\bigg\}z_m\bigg]z_1^{n+1-i}, \nonumber \\
			&V=\{n+{(-1)}^k\}\sum_{\a_1=2}^{k}z_{\a_1}z_{k+2-\a_1}z_1^{n-1},\nonumber \\
			&W=\{1+{(-1)}^k\}(n+1)z_{k+1} z_1^{n}.
		\end{align*}	
		Applying the antiautomorphism $\tau$, on both sides of \eqref{equn47}, we get
		
		\begin{align}
			\label{equn48}
			&	\ \ \ \ \ \hspace{1em} \sum_{r=0}^{k-2}{(-1)}^r \tau[z_{k-r}\shuffle z_{r+2}z_1^{n-1}]
			=\tau[U+V+W]=\tau[U]+\tau[V]+\tau[W] \nonumber \\
			&\implies \sum_{r=0}^{k-2}{(-1)}^r [\tau(x^{k-r-1}y) \shuffle \tau(x^{r+1}y^n)]=\tau[U]+\tau[V]+\tau[W] \nonumber \\
			&\implies \sum_{r=0}^{k-2}{(-1)}^r[xy^{k-r-1} \shuffle x^ny^{r+1}]=\tau[U]+\tau[V]+\tau[W] \nonumber \\
			&\implies \sum_{r=0}^{k-2}{(-1)}^r[z_{2}z_1^{k-r-2} \shuffle z_{n+1}z_1^{r}]=\tau[U]+\tau[V]+\tau[W].
		\end{align}
		
		Now,
		\begin{align}
			\label{equn49}
			\tau[U]&=\sum_{i=3}^{n+1} (n+2-i)\sum_{s=2}^{k}\tau \bigg[z_s \bigg[\sum_{m=2}^{k-s+2}\bigg\{\sum_{p=2}^{k-s-m+3}\sum_{j=0}^{i-4}z_{k-p-s-m+4}z_1^jz_{p}z_1^{i-4-j}+z_{k-s-m+3}z_1^{i-3} \nonumber\\
			& \ \ \ \ \ \  +\sum_{p=s+m-1}^{k-1}z_{k-p}\sum_{ \substack{\a_j \geq 1;\a_j \neq p-m-s+4; \\ \a_3+\dots +\a_{i-1}=p+i-m-s}} z_{\a_3} \dots z_{\a_{i-1}}\bigg\}z_m\bigg]z_1^{n+1-i}\bigg] \nonumber \\
			&=\sum_{i=3}^{n+1} (n+2-i)\sum_{s=2}^{k}\bigg[ \sum_{m=2}^{k-s+2}\bigg\{\sum_{p=2}^{k-s-m+3}\sum_{j=0}^{i-4}\tau \bigg(z_sz_{k-p-s-m+4}z_1^jz_{p}z_1^{i-4-j}z_mz_1^{n+1-i} \nonumber\\
			&  \ \ +z_sz_{k-s-m+3}z_1^{i-3}z_mz_1^{n+1-i}\bigg)  +\sum_{p=s+m-1}^{k-1}\sum_{ \substack{\a_j \geq 1;\a_j \neq p-m-s+4; \\ \a_3+\dots +\a_{i-1}=p+i-m-s}}\tau \bigg(z_sz_{k-p} z_{\a_3} \dots z_{\a_{i-1}}z_mz_1^{n+1-i}\bigg)\bigg\}\bigg].
		\end{align}	
		
	Note that,	
		\begin{align}
			\label{equn50}
			& \hspace*{1em} \tau \big(z_sz_{k-p-s-m+4}z_1^jz_{p}z_1^{i-4-j}z_{m}z_1^{n+1-i} + z_sz_{k-s-m+3}z_1^{i-3}z_{m}z_1^{n+1-i}\big) \nonumber \\
			&=\tau\big(x^{s-1}yx^{k-p-s-m+3}y^{j+1}x^{p-1}y^{i-j-3}x^{m-1}y^{n+2-i}+ x^{s-1}yx^{k-s-m+2}y^{i-2}x^{m-1}y^{n+2-i}\big) \nonumber \\
			&=x^{n+2-i}y^{m-1}x^{i-j-3}y^{p-1}x^{j+1}y^{k-p-s-m+3}xy^{s-1} + x^{n+2-i}y^{m-1}x^{i-2}y^{k-s-m+2}xy^{s-1} \nonumber \\
			&=z_{n+3-i}z_1^{m-2}z_{i-2-j}z_1^{p-2}z_{j+2}z_1^{k-p-s-m+2}z_2z_1^{s-2} + z_{n+3-i}z_1^{m-2}z_{i-1}z_1^{k-s-m+1}z_2z_1^{s-2}.
		\end{align}	
	
	and
	     \begin{align}
	     	\label{adi}
	     	&\tau \{z_s z_{k-p}z_{\a_3} \dots z_{\a_{i-1}}z_mz_1^{n+1-i}\} \nonumber \\
	     	& =\tau\{x^{s-1}yx^{k-p-1}yx^{\a_3-1}y \dots x^{\a_{i-1}-1}yx^{m-1}y^{n+2-i}\} \nonumber \\
	     	&  =\tau\{x^{n+2-i}y^{m-1}xy^{\a_{i-1}-1}x \dots xy^{\a_3-1}xy^{k-p-1}xy^{s-1}\}\nonumber \\
	     	& = z_{n+3-i}z_1^{m-2}z_2z_1^{\a_{i-1}-2}z_2 \dots z_2z_1^{\a_3-2}z_2z_1^{k-p-2}z_2z_1^{s-2}.
	     \end{align}
		
		Thus from \eqref{equn49}, \eqref{equn50} and \eqref{adi} , we get
		
		\begin{align}
			\label{equn51}
			\tau[U]&=\sum_{i=3}^{n+1} (n+2-i)z_{n+3-i}\sum_{s=2}^{k} \bigg[\sum_{m=2}^{k-s+2}z_1^{m-2}\bigg\{\sum_{p=2}^{k-s-m+3}\sum_{j=0}^{i-4} z_{i-2-j}z_1^{p-2}z_{j+2}z_1^{k-p-s-m+2} \nonumber\\
			&  \ \ + z_{i-1}z_1^{k-s-m+1}   +\sum_{p=s+m-1}^{k-1}\sum_{ \substack{\a_j \geq 1;\a_j \neq p-m-s+4; \\ \a_3+\dots +\a_{i-1}=p+i-m-s}}z_2z_1^{\a_{i-1}-2}z_2 \dots z_2z_1^{\a_3-2}z_2z_1^{k-p-2} \bigg\}\bigg]z_2z_1^{s-2}\nonumber\\
			&=\sum_{i=3}^{n+1} (n+2-i)z_{n+3-i}X,
		\end{align}
		where
		\begin{align*}
			X&=\sum_{s=2}^{k} \bigg[\sum_{m=2}^{k-s+2}z_1^{m-2}\bigg\{\sum_{p=2}^{k-s-m+3}\sum_{j=0}^{i-4} z_{i-2-j}z_1^{p-2}z_{j+2}z_1^{k-p-s-m+2} \nonumber\\
			&  \ \ + z_{i-1}z_1^{k-s-m+1}   +\sum_{p=s+m-1}^{k-1}\sum_{ \substack{\a_j \geq 1;\a_j \neq p-m-s+4; \\ \a_3+\dots +\a_{i-1}=p+i-m-s}}z_2z_1^{\a_{i-1}-2}z_2 \dots z_2z_1^{\a_3-2}z_2z_1^{k-p-2} \bigg\}\bigg]z_2z_1^{s-2}.
		\end{align*}
		
		Observe that $X$ is the sum of all words with depth $k-1$  and weight $k+i-2$, that is,
		\begin{align*}
			X&=\sum_{ \substack{\a_j \geq 1;\a_1+\a_2+ \dots +\a_{k-1}\\=k+i-2}} z_{\a_1} z_{\a_2}\dots z_{\a_{k-1}} \\
			&=\sum_{ \substack{a_j \geq 0;a_1+a_2+ \dots +a_{k-1}\\=k+i-2-{(k-1)}}} z_{a_1+1} z_{a_2+1}\dots z_{a_{k-1}+1}\\
			&=\sum_{ \substack{a_j \geq 0;a_1+a_2+ \dots +a_{k-1}\\=i-1}} z_{a_1+1} z_{a_2+1}\dots z_{a_{k-1}+1}.
		\end{align*}

		Therefore from \eqref{equn51}, 
		
		\begin{align}
			\label{equn52}
			\tau[U]&=\sum_{i=3}^{n+1} (n+2-i)z_{n+3-i}\sum_{ \substack{a_1+a_2+ \dots +a_{k-1}=i-1;\\a_j \geq 0}} z_{a_1+1} z_{a_2+1}\dots z_{a_{k-1}+1} \nonumber \\
			&={(n-1)}z_n\sum_{ \substack{a_1+a_2+ \dots +a_{k-1}=2;\\a_j \geq 0}} z_{a_1+1} z_{a_2+1}\dots z_{a_{k-1}+1} \nonumber \\
			& \ \ \ \ \ + {(n-2)}z_{n-1}\sum_{ \substack{a_1+a_2+ \dots +a_{k-1}=3;\\a_j \geq 0}} z_{a_1+1} z_{a_2+1}\dots z_{a_{k-1}+1}\nonumber \\
			& \ \ \ \ \ + \dots + 2z_{3}\sum_{ \substack{a_1+a_2+ \dots +a_{k-1}=n-1;\\a_j \geq 0}} z_{a_1+1} z_{a_2+1}\dots z_{a_{k-1}+1}\nonumber \\
			& \ \ \ \ \ +	z_{2}\sum_{ \substack{a_1+a_2+ \dots +a_{k-1}=n;\\a_j \geq 0}} z_{a_1+1} z_{a_2+1}\dots z_{a_{k-1}+1}\nonumber \\
			&=\sum_{ \substack{a_1+a_2+ \dots +a_{k}=n;\\a_j \geq 0; 0\leq a_k \leq n-2}} {(a_k+1)}z_{a_k+2} z_{a_1+1} z_{a_2+1}\dots z_{a_{k-1}+1}.
		\end{align}

		Now
		\begin{align}
			\tau [V]&= \tau \big[\{n+{(-1)}^k\}\sum_{\a_1=2}^{k}z_{\a_1}z_{k+2-\a_1}z_1^{n-1}\big] \nonumber \\
			&=\{n+{(-1)}^k\}\sum_{\a_1=2}^{k}\tau\big[x^{\a_1-1}yx^{k+1-\a_1}y^n\big] \nonumber \\
			&=\{n+{(-1)}^k\}\sum_{\a_1=2}^{k}\big[x^ny^{k+1-\a_1}xy^{\a_1-1} \big] \nonumber \\
			&=\{n+{(-1)}^k\}\sum_{\a_1=2}^{k}\big[z_{n+1}z_1^{k-\a_1}z_2z_1^{\a_1-2}\big].
		\end{align}
		
		and
		\begin{align}
			\tau [W]&=\tau\big[\{1+{(-1)}^k\}(n+1)z_{k+1} z_1^{n}\big] \nonumber \\
			&=\{1+{(-1)}^k\}(n+1)\tau\big[ x^ky^{n+1}\big] \nonumber \\
			&=\{1+{(-1)}^k\}(n+1)\big[x^{n+1} y^k\big] \nonumber \\
			&=\{1+{(-1)}^k\}(n+1)z_{n+2}z_1^{k-1}.
		\end{align}
		
		Using the above obtained expressions of $\tau[U]$, $\tau [V]$, and  $\tau[W]$ in \eqref{equn48}, we get
		
		\begin{align}
			\label{equn55}
			& \hspace*{1cm}\sum_{r=0}^{k-2}{(-1)}^r[z_{2}z_1^{k-r-2} \shuffle z_{n+1}z_1^{r}] \nonumber \\
			&=\sum_{ \substack{a_1+a_2+ \dots +a_{k}=n;\\a_j \geq 0; 0\leq a_k \leq n-2}} {(a_k+1)}z_{a_k+2} z_{a_1+1} z_{a_2+1}\dots z_{a_{k-1}+1} \nonumber \\
			& \ \ \ \ +\{n+{(-1)}^k\}\sum_{\a_1=2}^{k}\big[z_{n+1}z_1^{k-\a_1}z_2z_1^{\a_1-2}\big] + \{1+{(-1)}^k\}(n+1)z_{n+2}z_1^{k-1} \nonumber \\
			&=\sum_{ \substack{a_1+a_2+ \dots +a_{k}=n;\\a_j \geq 0 }} {(a_k+1)}z_{a_k+2} z_{a_1+1} z_{a_2+1}\dots z_{a_{k-1}+1} \nonumber \\
			& \ \ \ \ +{(-1)}^k\sum_{\a_1=2}^{k}\big[z_{n+1}z_1^{k-\a_1}z_2z_1^{\a_1-2}\big] + {(-1)}^k(n+1)z_{n+2}z_1^{k-1} \nonumber \\
			&=\sum_{ \substack{a_1+a_2+ \dots +a_{k}=n;\\a_j \geq 0 }} {(a_k+1)}z_{a_k+2} z_{a_1+1} z_{a_2+1}\dots z_{a_{k-1}+1} \nonumber \\
			& \ \ \ \ +{(-1)}^k \big[z_{n+1}z_2z_1^{k-2}+z_{n+1}z_1z_2z_1^{k-3} + \dots + z_{n+1}z_1^{k-2}z_2 \nonumber \\
			& \hspace{14em}+ {(n+1)}z_{n+2}z_1^{k-1} \big].
		\end{align}
		
		Hence the theorem.
	\end{proof}
	

	\begin{theorem}
		For any integers $k,n \geq 1$,
		\begin{align}
			\label{equn64}
			\zeta^*(k+1,\underbrace{1, \dots ,1}_{n})&={(-1)}^{k-1} \big[\zeta({n+1,\underbrace{2,1, \dots, 1}_{k-1}})
			+\zeta({n+1,\underbrace{1,2,1, \dots, 1}_{k-1}}) \nonumber \\
			& \ \ \ \ + \dots  + \zeta({n+1,\underbrace{1,1, \dots, 1,2}_{k-1}})  +(n+1)\zeta({n+2,\underbrace{1, \dots, 1}_{k-1}})\big]  \nonumber \\
			& \ \ \ \ + \sum_{r=0}^{k-2}{(-1)}^r\zeta({k-r})\zeta({n+1,\underbrace{1, \dots, 1}_{r}}).
		\end{align}
	\end{theorem}
	
	\begin{proof}

		Applying the map $Z$ on both sides of \eqref{equn55}, we get
		
		\begin{align}
			\label{equn57}
			&\sum_{r=0}^{k-2}{(-1)}^r\zeta({2,\underbrace{1, \dots, 1}_{k-r-2}})\zeta({n+1,\underbrace{1, \dots, 1}_{r}}) \nonumber \\
			&=\sum_{ \substack{a_1+a_2+ \dots +a_{k}=n\\a_j \geq 0 }} {(a_k+1)}\zeta({a_k+2,a_1+1,a_2+1,\dots, a_{k-1}+1}) \nonumber \\
			& \ \ \ \ + {(-1)}^{k} \big[\zeta({n+1,\underbrace{2,1, \dots, 1}_{k-1}})
			+\zeta({n+1,\underbrace{1,2,1, \dots, 1}_{k-1}})  \nonumber \\
			& \ \ \ \ + \dots + \zeta({n+1,\underbrace{1,1, \dots, 1,2}_{k-1}})  +(n+1)\zeta({n+2,\underbrace{1, \dots, 1}_{k-1}})\big].
		\end{align}
		
		Since the indices $({2,\underbrace{1, \dots, 1}_{k-r-2}})$ and $(k-r)$ are dual to each other, by  Theorem 1.2,
		$$\zeta({2,\underbrace{1, \dots, 1}_{k-r-2}})=\zeta(k-r).$$
		
	 Using this in \eqref{equn57}, we get
		
		\begin{align}
			\label{equn58}
			&\hspace{5em}\sum_{r=0}^{k-2}{(-1)}^r\zeta({k-r})\zeta({n+1,\underbrace{1, \dots, 1}_{r}}) \nonumber \\
			& \hspace{5em}=\sum_{ \substack{a_1+a_2+ \dots +a_{k}=n\\a_j \geq 0 }} {(a_k+1)}\zeta({a_k+2,a_1+1,a_2+1,\dots, a_{k-1}+1}) \nonumber \\
			& \ \ \ \hspace{6em} + {(-1)}^{k} \big[\zeta({n+1,\underbrace{2,1, \dots, 1}_{k-1}})
			+\zeta({n+1,\underbrace{1,2,1, \dots, 1}_{k-1}}) \nonumber \\
			& \ \ \ \ \hspace{6em} + \dots  + \zeta({n+1,\underbrace{1,1, \dots, 1,2}_{k-1}})
			+(n+1)\zeta({n+2,\underbrace{1, \dots, 1}_{k-1}})\big] \nonumber \\
			& \implies \sum_{ \substack{a_1+a_2+ \dots +a_{k}=n\\a_j \geq 0 }} {(a_k+1)}\zeta({a_k+2,a_1+1,a_2+1,\dots, a_{k-1}+1}) \nonumber \\
			& \hspace{5em}={(-1)}^{k-1} \big[\zeta({n+1,\underbrace{2,1, \dots, 1}_{k-1}})
			+\zeta({n+1,\underbrace{1,2,1, \dots, 1}_{k-1}}) \nonumber \\
			& \ \ \ \ \hspace{5em} + \dots  + \zeta({n+1,\underbrace{1,1, \dots, 1,2}_{k-1}})
			+(n+1)\zeta({n+2,\underbrace{1, \dots, 1}_{k-1}})\big]  \nonumber \\
			& \ \ \ \ \hspace{5em} + \sum_{r=0}^{k-2}{(-1)}^r\zeta({k-r})\zeta({n+1,\underbrace{1, \dots, 1}_{r}}).
		\end{align}
		Hence from Proposition \ref{p}, the theorem follows.
		
	\end{proof}

\textbf{Proof of Theorem \ref{T105}:}	
	
		Since the indices $(k+1,\underbrace{1, \dots ,1}_{n-1})$ and $(n+1,\underbrace{1, \dots ,1}_{k-1})$ are dual to each other, by Theorem \ref{ohn},
		\begin{align}
			\label{e63}
			& \zeta(k+2,\underbrace{1, \dots ,1}_{n-1})+ \zeta(k+1,\underbrace{2,1, \dots ,1}_{n-1})+\zeta(k+1,\underbrace{1, 2, 1,  \dots ,1}_{n-1})+ \dots + \zeta(k+1,\underbrace{1, \dots ,1, 2}_{n-1}) \nonumber\\
			& =\zeta(n+2,\underbrace{1, \dots ,1}_{k-1})+ \zeta(n+1,\underbrace{2,1, \dots ,1}_{k-1})+\zeta(n+1,\underbrace{1, 2, 1,  \dots ,1}_{k-1})+ \dots + \zeta(n+1,\underbrace{1, \dots ,1, 2}_{k-1}).
		\end{align}
Therefore, from \eqref{equn64} and \eqref{e63}, one can easily obtain the theorem.\\


$\mathbf{Acknowledgement:}$ The research of first author is supported by the University assistance/support for research, NBU (Ref. No. 2302/ R-2022) and that of second author is supported by the University Grants Commission (UGC), India through NET-JRF (Ref. No. 191620198830).

	{}

\end{document}